\documentclass{article}%
\usepackage{graphicx}
\usepackage{amsmath}
\usepackage{amsfonts}
\usepackage{amssymb}%
\providecommand{\U}[1]{\protect\rule{.1in}{.1in}}
\newtheorem{theorem}{Theorem}
{}

\newtheorem{corollary}{Corollary}

\newtheorem{example}{Example}

\newtheorem{lemma}{Lemma}
{}

\newtheorem{proposition}{Proposition}
\newtheorem{remark}{Remark}

\newenvironment{proof}[1][Proof]{\textbf{#1.} }{\ \rule{0.5em}{0.5em}}

\begin{document}

\title{On Exact Estimates of Instability Zones of the Hill's Equation with Locally
Integrable Potential}
\author{O.A.Veliev\\{\small Dogus University, \ Istanbul, Turkey.}\\\ {\small e-mail: oveliev@dogus.edu.tr}}
\date{}
\maketitle

\begin{abstract}
In this paper we consider the one-dimensional Schr\"{o}dinger operator $L(q)$
with a periodic real and locally integrable potential $q$. We study the bands
and gaps in the spectrum and explicitly write out the first and second terms
of the asymptotic formulas for the length of the gaps in the spectrum.



\end{abstract}

\section{Introduction and Preliminary Facts}

In this paper we investigate the spectrum of the operator $L(q)$ generated in
$L_{2}(-\infty,\infty)$ by the differential expression
\begin{equation}
l(y)=-y^{^{\prime\prime}}+qy, \tag{1}%
\end{equation}
where $q$ is $1$-periodic integrable on $[0,1]$ and a real-valued potential.
Without loss of generality, it is assumed that
\begin{equation}
\int_{0}^{1}q(x)dx=0. \tag{2}%
\end{equation}

It is well-known that (see for example [2-4, 6, 7, 9, 11]) the spectrum
$\sigma(L(q))$ of $L(q)$ is the union of the spectra $\sigma(L_{t}(q))$ of the
operators $L_{t}(q)$ for $t\in(-\pi,\pi]$ generated in $L_{2}[0,1]$ by (1) and
the boundary conditions
\begin{equation}
y(1)=e^{it}y(0),\text{ }y^{^{\prime}}(1)=e^{it}y^{^{\prime}}(0). \tag{3}%
\end{equation}
Moreover, $\sigma(L(q))$ consists of the closed intervals whose end points are
the eigenvalues of $L_{0}(q)$ and $L_{\pi}(q).$ Therefore, to study the
spectrum of $L(q)$ it is enough to investigate the eigenvalues of $L_{0}(q)$
and $L_{\pi}(q),$ which are called periodic and antiperiodic eigenvalues,
respectively. By the classical investigations (see formulas (47a), (47b) in
page 65 of the monograph [10]) the large eigenvalues of the operators
$L_{0}(q)$ and $L_{\pi}(q)$ consist of the sequences $\{\lambda_{n,j}%
:n=N,N+1,...\}$ and $\{\mu_{n,j}:n=N,N+1,...\}$\ satisfying
\begin{equation}
\lambda_{n,j}=(2n\pi)^{2}+O(n^{\frac{1}{2}}),\text{ } \tag{4}%
\end{equation}
and
\[
\mu_{n,j}=(2n\pi+\pi)^{2}+O(n^{\frac{1}{2}})
\]
for $j=1,2$, where $N$ is a sufficiently large positive integer. Thus, the
gaps in $\sigma(L(q))$ in the high energy region are equal to
\begin{equation}
\Delta_{n}=(\lambda_{n,1},\lambda_{n,2}),~\Omega_{n}=\left(  \mu_{n,1}%
,\mu_{n,2}\right)  \tag{5}%
\end{equation}
for $n=N,N+1,...,$ and without loss of generality it is assumed that
$\lambda_{n,1}\leq\lambda_{n,2}~$and $\mu_{n,1}\leq\mu_{n,2}.$ In this paper
to obtain a sharp and explicit formulas for the lengths of these gaps, we use
some formulas from [1] and [12]. For the independent reading of this paper,
first of all, we list the formulas of those papers that are essentially used here.

In [1] we obtained asymptotic formulas of arbitrary order for the eigenvalues
$\lambda_{n,j}$ and corresponding normalized eigenfunctions $\Psi_{n,j}(x)$ of
$L_{0}(q),$ for any $q\in L_{1}[0,1].$ For this we used (4) and the equality
\begin{equation}
(\lambda_{n,j}-(2\pi n)^{2})(\Psi_{n,j}(x),e^{i2\pi nx})=(q(x)\Psi
_{n,j}(x),e^{i2\pi nx})= \tag{6}%
\end{equation}%
\[
\sum_{\substack{n_{1}=-\infty,\\n_{1}\neq0}}^{\infty}q_{n_{1}}(\Psi
_{n,j},e^{i2\pi(n-n_{1})x}),
\]
where $(\cdot,\cdot)$ is the inner product in $L_{2}[0,1]$ and $q_{n}%
=(q(x),e^{i2\pi nx}).$ Iterating formula (6) we obtained
\begin{equation}
(\lambda_{n,j}-(2\pi n)^{2}-A_{m}(\lambda_{n,j}))(\Psi_{n,j}(x),e^{i2\pi nx})-
\tag{7}%
\end{equation}%
\[
(q_{2n}+B_{m}(\lambda_{n,j}))(\Psi_{n,j}(x),e^{-i2\pi nx})=R_{m}(\lambda
_{n,j}),
\]
where
\[
A_{m}(\lambda_{n,j})=\sum_{k=1}^{m}a_{k}(\lambda_{n,j}),\text{ }B_{m}%
(\lambda_{n,j})=\sum_{k=1}^{m}b_{k}(\lambda_{n,j}),
\]%
\[
a_{k}(\lambda_{n,j})=\sum_{n_{1},n_{2},...,n_{k}}\frac{q_{n_{1}}q_{n_{2}%
}...q_{n_{k}}q_{-n_{1}-n_{2}-...-n_{k}}}{\prod\limits_{s=1}^{k}(\lambda
_{n,j}-(2\pi(n-n_{1}-n_{2}-...-n_{s}))^{2})},
\]%
\[
b_{k}(\lambda_{n,j})=\sum_{n_{1},n_{2},...,n_{k}}\frac{q_{n_{1}}q_{n_{2}%
}...q_{n_{k}}q_{2n-n_{1}-n_{2}-...-n_{k}}}{\prod\limits_{s=1}^{k}%
(\lambda_{n,j}-(2\pi(n-n_{1}-n_{2}-...-n_{s}))^{2})},
\]%
\[
R_{m}(\lambda_{n,j})=\sum_{n_{1},n_{2},...,n_{m+1}}\frac{q_{n_{1}}q_{n_{2}%
}...q_{n_{m}}q_{n_{m+1}}(q(x)\Psi_{n,j}(x),e^{i2\pi(n-n_{1}-...-n_{m+1})x}%
)}{\prod\limits_{s=1}^{m+1}(\lambda_{n,j}-(2\pi(n-n_{1}-n_{2}-...-n_{s}%
))^{2})}.
\]
The summation in these formulas is taken over the indices satisfying the
conditions%
\begin{equation}
n_{s}\neq0,\text{ }n_{1}+n_{2}+...,n_{s}\neq0,2n \tag{8}%
\end{equation}
for $s=1,2,...,m+1.$

Moreover, in [1] it were proved that
\begin{equation}
a_{k}(\lambda_{n,j})=O((\frac{\ln\left\vert n\right\vert }{n})^{k}),\text{
}b_{k}(\lambda_{n,j})=O((\frac{\ln\left\vert n\right\vert }{n})^{k}) \tag{9}%
\end{equation}
for $k\geq1$ and
\begin{equation}
R_{m}(\lambda_{n,j})=O((\frac{\ln\left\vert n\right\vert }{n})^{m+1}).
\tag{10}%
\end{equation}
Using these estimations in (7) we obtained
\begin{equation}
(\lambda_{n,j}-(2\pi n)^{2})u_{n,j}=q_{2n}u_{-n,j}+O\left(  \frac
{\ln\left\vert n\right\vert }{n}\right)  , \tag{11}%
\end{equation}
where $u_{n,j}=(\Psi_{n,j}(x),e^{i2\pi nx})$. In the same way we have got
\begin{equation}
(\lambda_{n,j}-(2\pi n)^{2})u_{-n,j}=q_{-2n}u_{n,j}+O\left(  \frac
{\ln\left\vert n\right\vert }{n}\right)  . \tag{12}%
\end{equation}

Besides, we obtained the following obvious and well known formula (see (19)
and (20) of [1] or Theorem 2.3.2 in page 50 of [14])
\begin{equation}
\Psi_{n,j}(x)=u_{n,j}e^{i2\pi nx}+u_{-n,j}e^{-i2\pi nx}+h(x), \tag{13}%
\end{equation}
where $\left\Vert h(x)\right\Vert =O(n^{-1}),$ $(h,e^{\pm i2\pi nx})=0$ and
\[
\left\vert u_{n,j}\right\vert ^{2}+\left\vert u_{-n,j}\right\vert
^{2}=1+O(n^{-2}).
\]
From (11)-(13) it readily follows that the sequence $\{\lambda_{n,j}%
\}_{N}^{\infty}\ $ (see (4)) obey an asymptotic formula
\begin{equation}
\lambda_{n,j}=(2n\pi)^{2}+o(1) \tag{14}%
\end{equation}
as$\ \,n\rightarrow\infty$ for $j=1,2.$ (If $\left\vert u_{n,j}\right\vert
\geq\left\vert u_{-n,j}\right\vert $ then use (11), if $\left\vert
u_{-n,j}\right\vert \geq\left\vert u_{n,j}\right\vert $ then use (12)).
Similarly, for $\{\mu_{n,1}\}_{N}^{\infty}\ $ we have a formula
\[
\mu_{n,j}=(2n\pi+\pi)^{2}+o(1).
\]

Finally, note that we use the following equalities from [12]. Let $p\geq0$ be
an arbitrary integer,
\begin{equation}
q\in W_{1}^{p}[0,1],\,q^{(l)}(0)=q^{(l)}(1) \tag{15}%
\end{equation}
for all$\ \,0\leq l\leq s-1,$ with some $s\leq p,$ where $W_{1}^{p}[0,1]$ is
the Sobolev space and $W_{1}^{0}[0,1]=L_{1}[0,1]$. Define the functions
\[
Q(x)=\int_{0}^{x}q(t)\,dt,\quad S(x)=Q^{2}(x),
\]
and denote by $Q_{k}=(Q,\,e^{2\pi ikx}),\ \,S_{k}=(S,\,e^{2\pi ikx})$ the
Fourier coefficients of these functions. The following relations are valid:
\begin{equation}
b_{1}(\lambda_{n,j})=-S_{2n}+2Q_{0}Q_{2n}+o\left(  n^{-s-2}\right)
,\quad\tag{16}%
\end{equation}
and%
\begin{equation}
b_{2}(\lambda_{n,j})=o\left(  n^{-s-2}\right)  ,\text{ }b_{k}(\lambda
_{n})=o\left(  \frac{\ln^{k}\,n}{n^{k+s}}\right)  \tag{17}%
\end{equation}
for all$\ k=3,4,\dots$ (see Lemma 6 of [12]).

In this paper we obtain asymptotic formulas
\begin{equation}
\lambda_{n,j}=(2\pi n)^{2}+(-1)^{j}\left\vert q_{2n}\right\vert +o(n^{-1})
\tag{18}%
\end{equation}
and
\begin{equation}
\lambda_{n,j}=(2\pi n)^{2}+A_{2}((2\pi n)^{2})+(-1)^{j}\left\vert
q_{2n}-S_{2n}+2Q_{0}Q_{2n}\right\vert +o(n^{-2}) \tag{19}%
\end{equation}
for the sequence $\{\lambda_{n,j}\}_{N}^{\infty}.$ Moreover, from these
formulas we get the following formulas for the lengths $\left\vert \Delta
_{n}\right\vert $ of the gaps $\Delta_{n}$%
\begin{equation}
\left\vert \Delta_{n}\right\vert =2\left\vert q_{2n}\right\vert +o(n^{-1})
\tag{20}%
\end{equation}
and $\ $
\begin{equation}
\left\vert \Delta_{n}\right\vert =2\left\vert q_{2n}-S_{2n}+2Q_{0}%
Q_{2n}\right\vert +o(n^{-2}). \tag{21}%
\end{equation}
The similar formulas for for the sequence $\{\mu_{n,j}\}_{N}^{\infty}$ and
$\left\vert \Omega_{n}\right\vert $ are obtained. The superiority of formula
(20) over well-known Titchmarsh formula
\begin{equation}
\left\vert \Delta_{n}\right\vert =2\left\vert q_{2n}\right\vert +O(n^{-1})
\tag{22}%
\end{equation}
(see [13] Chap. 21, [7] Chap. 3, and their references) for any potential $q\in
L_{1}[0,1]$ can be explained as follows. Generally speaking, some Fourier
coefficient $q_{2n}$ of many functions from $L_{1}[0,1]$ can be of order
$n^{-1}$ in the sense that $q_{2n}=O(n^{-1})$ and $n^{-1}=O(q_{2n})$ when $n$
belong to some subset of $\mathbb{N}$ (for example if the potential $q$ has
some jump discontinuity). Then (20) with the error term $o(n^{-1})$ explicitly
separate the term $2\left\vert q_{2n}\right\vert $ from the error $o(n^{-1}),$
while (22) does not do this. In this sense $2\left\vert q_{2n}\right\vert $
can be seen as first term of the asymptotic formula for the gaps in
$\sigma(L(q))$ with the potential $q\in L_{1}[0,1].$ Similarly, (21) give us
the second term $2\left\vert q_{2n}-S_{2n}+2Q_{0}Q_{2n}\right\vert
-2\left\vert q_{2n}\right\vert $ of the asymptotic formula for the gaps in
$\sigma(L(q)).$

Then we obtain asymptotic formulas with the error term $O\left(  (\frac{\ln
n}{n})^{k}\right)  $ for the $n$th\ periodic and antiperiodic eigenvalues and
$n$th gaps in the spectrum of the operator $L(q)$ with $q\in L_{1}[0,1],$
where $k=3,4,...$. Many studies are devoted to these asymptotic formulas for
differentiable potentials (see [8, Chapter 1] and [7, Chapter 3] and their
references). Here we consider only the potential $q$ from $L_{1}[0,1]$.
Finally, we take the Kronig-Penney model as an example and illustrate formulas
(20) and (21) with this example.

\section{Main Results}

In this section we obtain asymptotic formulas for the eigenvalues and
eigenfunctions of the operators $L_{0}(q)$ and $L_{\pi}(q)$ and for the gaps
in $\sigma(L(q)).$ First, we prove the following obvious corollary of (14).

\begin{proposition}
If $q\in L_{1}[0,1],$ then the following equalities hold%
\[
a_{1}(\lambda_{n,j})=a_{1}((2\pi n)^{2})+o(n^{-2}),
\]%
\[
a_{2}(\lambda_{n,j})=a_{2}((2\pi n)^{2})+o(n^{-3}\ln n)
\]
and%
\[
b_{1}(\lambda_{n,j})=b_{1}((2\pi n)^{2})+o(n^{-2}),
\]
where $a_{1}((2\pi n)^{2}),$ $a_{2}((2\pi n)^{2})$ and $b_{1}((2\pi n)^{2})$
are obtained respectively from $a_{1}(\lambda_{n,j}),$ $a_{2}(\lambda_{n,j})$
and $b_{1}(\lambda_{n,j})$ by changing $\lambda_{n,j}$ to $(2\pi n)^{2}$ in
the corresponding formulas.
\end{proposition}

\begin{proof}
From (14) it follows that
\[
\frac{q_{n_{1}}q_{-n_{1}}}{\lambda_{n,j}-(2\pi(n-n_{1}))^{2}}-\frac{q_{n_{1}%
}q_{-n_{1}}}{(2\pi n)^{2}-(2\pi(n-n_{1}))^{2}}=
\]%
\[
\frac{o(1)}{((2\pi n)^{2}-(2\pi(n-n_{1}))^{2})^{2}}=\frac{o(1)}{(n_{1}%
(2n-n_{1}))^{2}}.
\]
Now, using this and the obvious equality
\[
\sum_{\substack{n_{1}\neq0,2n}}^{\infty}\frac{1}{(n_{1}(2n-n_{1}))^{2}%
}=O(n^{-2})
\]
we obtain a proof of the first equality of the proposition. In a similar way
we obtain the proof of the third equality.

Now, let's prove the second equality. Using (14), it is easy to see that
\[
\frac{q_{n_{1}}q_{n_{2}}q_{-n_{1}-n_{2}}}{[\lambda_{n,j}-(2\pi(n-n_{1}%
))^{2}][\lambda_{n,j}-(2\pi(n-n_{1}-n_{2}))^{2}]}-
\]%
\[
\frac{q_{n_{1}}q_{n_{2}}q_{-n_{1}-n_{2}}}{[(2\pi n)^{2}-(2\pi(n-n_{1}%
))^{2}][(2\pi n)^{2}-(2\pi(n-n_{1}-n_{2}))^{2}]}=
\]%
\[
\frac{o(1)}{n_{1}(2n-n_{1})((n_{1}+n_{2})(2n-n_{1}-n_{2}))^{2}}+
\]%
\[
\frac{o(1)}{(n_{1}(2n-n_{1}))^{2}(n_{1}+n_{2})(2n-n_{1}-n_{2})},
\]
which provides a proof of the second equality.
\end{proof}

Now, let's estimate $a_{1}(\lambda_{n,j})$ in detail

\begin{lemma}
$(a)$ If $q\in L_{1}[0,1],$ then the following estimation holds%
\begin{equation}
a_{1}(\lambda_{n,j})=o(n^{-1}). \tag{23}%
\end{equation}

$(b)$ For each real-valued $q\in L_{1}[0,1]$ the expressions $a_{1}%
(\lambda_{n,j})$, $a_{2}(\lambda_{n,j}),...$ are the real numbers.
\end{lemma}

\begin{proof}
$(a)$ By Proposition 1 we need to prove that
\[
a_{1}((2\pi n)^{2})=o(n^{-1}).
\]
Since
\[
q_{-k}=\int\limits_{0}^{1}q(x)e^{i2\pi kx}dx=\overline{q_{k}},\text{
}\left\vert q_{k}\right\vert =\left\vert q_{-k}\right\vert ,q_{0}=0,
\]
grouping the terms with indices $k$ and $-k$ we obtain
\begin{equation}
a_{1}((2\pi n)^{2})=\frac{1}{4\pi^{2}}\sum_{\substack{k=-\infty\\k\neq
0,2n}}^{\infty}\frac{\left\vert q_{k}\right\vert ^{2}}{2nk-k^{2}}= \tag{24}%
\end{equation}%
\[
\frac{1}{4\pi^{2}}\sum_{k\in\mathbb{N},k\neq2n}\frac{\left\vert q_{k}%
\right\vert ^{2}}{2nk-k^{2}}-\frac{\left\vert q_{k}\right\vert ^{2}}%
{2nk+k^{2}}=
\]%
\[
=\frac{1}{2\pi^{2}}\sum_{k\in\mathbb{N},k\neq2n}\frac{\left\vert
q_{k}\right\vert ^{2}}{\left(  2n-k\right)  (2n+k)}.
\]
Now to prove (23) we show that the following two equalities hold
\begin{equation}
\frac{-1}{2\pi^{2}}\sum_{k\in\mathbb{N},k\neq2n}\frac{\left\vert
q_{k}\right\vert ^{2}}{\left(  2n-k\right)  (2n+k)}=\int_{0}^{1}%
(Q(x,n)-Q_{0})^{2}e^{-i8\pi nx}\,dx \tag{25}%
\end{equation}
and
\begin{equation}
\int_{0}^{1}(Q(x,n)-Q_{0})^{2}e^{-i8\pi nx}\,dx=o(n^{-1}), \tag{26}%
\end{equation}
where%
\[
Q(x,n)=\int_{0}^{x}q(t)e^{i4\pi nt}\,dt-q_{-2n}x.
\]
First let us prove (25) by calculation its right-hand side . It is clear that
\begin{equation}
Q(1,n)=Q(0,n)=0,\text{ }Q^{^{\prime}}(x,n)=q(x)e^{i4\pi nx}\,-q_{-2n}.
\tag{27}%
\end{equation}
Therefore, using the integrations by parts we see that the Fourier
coefficients $Q_{2n+k}=(Q(x,n),e^{i2\pi(2n+k)x})$ of $Q(x,n)$ for $2n+k\neq
0$\ is%
\[
Q_{2n+k}=\frac{q_{k}}{i2\pi(2n+k)}.
\]
Moreover, $q_{0}=0$ according to (2). Thus, using the Fourier decomposition
\[
Q(x,n)-Q_{0}=\sum_{\substack{k=-\infty\\k\neq0,-2n}}^{\infty}\frac
{q_{k}e^{i2\pi(2n+k)x}\,}{i2\pi(2n+k)}%
\]
of $Q(x,n)-Q_{0}$ in the right side of (25) we get its proof.

To prove (26), first we show that $Q(x,n)$ converges uniformly for
$x\in\lbrack0,1]$ to zero as $n\rightarrow\infty.$ Since $q\in L_{1}[0,1],$
for each $\varepsilon>0,$ there exists continuously differentiable function
$f_{\varepsilon}$ such that
\[
\int_{0}^{x}\left\vert q(t)-f_{\varepsilon}(t)\right\vert dt\leq\int_{0}%
^{1}\left\vert q(t)-f_{\varepsilon}(t)\right\vert dt<\varepsilon.
\]
On the other hand, using integration by parts we obtain that, there exists $N$
such that
\[
\left\vert \int_{0}^{x}f_{\varepsilon}(t)e^{i4\pi nt}\,dt\right\vert
\,<\varepsilon,
\]
for $x\in\lbrack0,1]$ and $n>N.$ From the last two inequalities it follows
that $Q(x,n)$ converges uniformly for $x\in\lbrack0,1]$ to zero as
$n\rightarrow\infty.$ Now, we are ready to prove (26). Using integration by
parts and (27) we obtain
\[
\int_{0}^{1}(Q(x,n)-Q_{0})^{2}e^{-i8\pi nx}\,dx=\frac{1}{^{i8\pi nt}}\int
_{0}^{1}2(Q(x,n)-Q_{0})(q(x)e^{i4\pi nx}\,-q_{-2n})e^{-i8\pi nx}\,dx=
\]%
\[
\frac{1}{^{i4\pi nt}}\int_{0}^{1}Q(x,n)(q(x)e^{i4\pi nx}\,-q_{-2n})e^{-i8\pi
nt}\,dt-
\]%
\[
\frac{Q_{0}(n)}{^{i4\pi nt}}\int_{0}^{1}(q(x)e^{i4\pi nx}\,-q_{-2n})e^{-i8\pi
nt}\,dt=o(n^{-1}).
\]
Thus (23) follows from (24)-(26).

$(b)$ Using the definition of $a_{k}(\lambda_{n,j})$, the obvious equality
$\overline{q_{n}}=q_{-n}$ and then replacing $n_{s}$ by $-n_{s}$we see that
\[
\overline{a_{k}(\lambda_{n,j})}=\sum_{n_{1},n_{2},...,n_{k}}\frac{q_{n_{1}%
}q_{n_{2}}...q_{n_{k}}q_{-n_{1}-n_{2}-...-n_{k}}}{\prod\limits_{s=1}%
^{k}(\lambda_{n,j}-(2\pi(n+n_{1}+n_{2}+...+n_{s}))^{2})},
\]
where the summation is taken over the indices $n_{1}+n_{2}+...,n_{s}\neq0,-2n$
and $n_{s}\neq0$ for $s=1,2,...,k.$ Now make the substitution
\[
-n_{1}-n_{2}-\dots-n_{k}=j_{1},\ \,n_{2}=j_{k},\ \,n_{3}=j_{k-1}%
,\,\dots,\,n_{k}=j_{2},
\]
in the formula for the expression $\overline{a_{k}(\lambda_{n,j})}.$ Then the
inequalities for the indices in the formula for $\overline{a_{k}(\lambda
_{n,j})}$ take the form $j_{s}\neq0,$ $j_{1}+j_{2}+...+j_{s}\neq0,2n$ and it
will coincide with the formula for $a_{k}(\lambda_{n,j})$ (see (7) and (8)).
The lemma is proved.
\end{proof}

Now let us discuss the eigenfunction $\Psi_{n,j}$ for the real potential.
Since $q(x)$ and $\lambda_{n,j}$ are the real numbers $\overline{\Psi_{n,j}}$
is also an eigenfunction corresponding to the eigenvalue $\lambda_{n,j}.$ Then
at least one of the functions $\Psi_{n,j}+\overline{\Psi_{n,j}}$ and
$i(\Psi_{n,j}-\overline{\Psi_{n,j}})$ is a real-valued eigenfunction.
Therefore, without loss of generality, assume that $\Psi_{n,j}$ is a real
-valued eigenfunction. Then
\[
u_{-n,j}=(\Psi_{n,j}(x),e^{-i2\pi nx})=\overline{(\Psi_{n,j}(x),e^{i2\pi nx}%
)}=\overline{u_{n,j}}.
\]
Using these notation in (13) we obtain that
\begin{equation}
\Psi_{n,j}(x)=u_{n,j}e^{i2\pi nx}+\overline{u_{n,j}}e^{-i2\pi nx}+h_{n,j}(x),
\tag{28}%
\end{equation}
where $(h_{n,j},e^{\pm i2\pi nx})=0,$ $\left\Vert h_{n,j}\right\Vert
=O(n^{-1})$ and%
\begin{equation}
\text{ }\left\vert u_{n,j}\right\vert =\frac{\sqrt{2}}{2}+O(n^{-2}). \tag{29}%
\end{equation}

Now, we are ready to obtain asymptotic formulas for the eigenvalues of the
operators $L_{0}(q)$ and $L_{\pi}(q)$\ for $q\in L_{1}(0,1).$ From (9), (10)
and (23) we obtain that (7) for $m=1$\ can be written in the form
\begin{equation}
(\lambda_{n,j}-(2\pi n)^{2})u_{n,j}=q_{2n}\overline{u_{n,j}}+\alpha_{n,j}
\tag{30}%
\end{equation}
where $\alpha_{n,j}=o(n^{-1}).$ Using these equalities we prove the following theorem.

\begin{theorem}
If $q\in L_{1}[0,1],$ then

$(a)$ The eigenvalues $\lambda_{n,1}$ and $\lambda_{n,2}$ satisfy the
equalities
\[
\lambda_{n,j}-(2\pi n)^{2}=a_{j}\left\vert q_{2n}\right\vert +O(\alpha_{n,j})
\]
for $j=1,2$ as $n\rightarrow\infty,$ where $a_{j}$ is either $1$ or $-1$.

$(b)$ If
\begin{equation}
\left\vert q_{2n}\right\vert \gg\left\vert \alpha_{n,1}\right\vert +\left\vert
\alpha_{n,2}\right\vert \tag{31}%
\end{equation}
for $n\geq N,$ then the eigenvalues \ $\lambda_{n,j}$ for $j=1,2$\ and $n\geq
N$ are simple and satisfy the asymptotic formulas
\begin{equation}
\lambda_{n,j}=(2\pi n)^{2}+(-1)^{j}\left\vert q_{2n}\right\vert +o\left(
\frac{1}{n}\right)  \tag{32}%
\end{equation}
as $n\rightarrow\infty.$
\end{theorem}

\begin{proof}
$(a)$ Dividing both sides of (30) by $u_{n,j}$ and using (29) we obtain%
\[
(\lambda_{n,j}-(2\pi n)^{2})-q_{2n}\overline{u_{n,j}}\frac{1}{u_{n,j}}%
=\frac{\alpha_{n,j}}{u_{n,j}}=O\left(  \alpha_{n,j}\right)  .
\]
This equality with the equality $\left\vert \overline{u_{n,j}}\frac{1}%
{u_{n,j}}\right\vert =1$ gives
\[
\left\vert \left\vert (\lambda_{n,j}-(2\pi n)^{2})\right\vert -\left\vert
q_{2n}\right\vert \right\vert =\left\vert \left\vert (\lambda_{n,j}-(2\pi
n)^{2})\right\vert -\left\vert q_{2n}\overline{u_{n,j}}\frac{1}{u_{n,j}%
}\right\vert \right\vert \leq
\]%
\[
\leq\left\vert (\lambda_{n,j}-(2\pi n)^{2})-q_{2n}\overline{u_{n,j}}\frac
{1}{u_{n,j}}\right\vert =O(\alpha_{n,j})
\]
which implies the proof of $(a)$, since $\lambda_{n,j}-(2\pi n)^{2}$ is a real number.

$(b)$ First we prove that if (31) hold, then $a_{1}\neq a_{2}.$ Suppose
$a_{1}=a_{2}$ $=-1$. Then, using $(a)$ in (30) for $j=1$ and $j=2,$ we obtain
\[
-\left\vert q_{2n}\right\vert u_{n,1}=q_{2n}\overline{u_{n,1}}+O(\alpha
_{n,1})
\]
and
\[
-\left\vert q_{2n}\right\vert u_{n,2}=q_{2n}\overline{u_{n,2}}+O(\alpha
_{n,2}).
\]
Multiplying both sides of the first and second equalities by $-\overline
{u_{n,2}}$ and $\overline{u_{n,1}}$ respectively and then summing these
equalities we get
\[
(u_{n,1}\overline{u_{n,2}}-u_{n,2}\overline{u_{n,1}})\left\vert q_{2n}%
\right\vert =-O(\alpha_{n,1})\overline{u_{n,2}}+O(\alpha_{n,2})\overline
{u_{n,1}}.
\]
Now dividing both sides of this equality by $\left\vert q_{2n}\right\vert $
and using (31) and (29) we obtain
\begin{equation}
\left\vert u_{n,1}\overline{u_{n,2}}-u_{n,2}\overline{u_{n,1}}\right\vert
<\frac{3}{4}. \tag{33}%
\end{equation}
On the other hand, if $\lambda_{n,1}\neq\lambda_{n,2}$ ,\ then the
corresponding eigenfunctions $\Psi_{n,1}$ and $\Psi_{n,2}$ are orthogonal. If
$\lambda_{n,1}=\lambda_{n,2}$ ,\ then $\Psi_{n,1}$ and $\Psi_{n,2}$ can be
chosen to be orthogonal. \ Therefore, using (28) we obtain
\begin{equation}
0=(\Psi_{n,1},\Psi_{n,2})=u_{n,1}\overline{u_{n,2}}+u_{n,2}\overline{u_{n,1}%
}+O(n^{-2}). \tag{34}%
\end{equation}
Now, from (33) and (34) it follows that $\left\vert 2u_{n,1}\overline{u_{n,2}%
}\right\vert <\frac{4}{5}$ which contradicts (29).

If we suppose that $a_{1}=a_{2}$ $=1$ then, in the same way, we obtain the
same contradiction. Thus one of these numbers is $1$ and the other is $-1$.
Since we assumed that $\lambda_{n,1}\leq\lambda_{n,2}$ (see (5)) we have
$a_{1}=-1$ and $a_{2}=1$ due to $(a)$ and (31). Thus, (32) is proved.
\end{proof}

Let us now prove that in (31), a much larger ratio can be replaced by twice as large.

\begin{theorem}
If $q\in L_{1}[0,1]$ and
\begin{equation}
\left\vert q_{2n}\right\vert \geq2(\left\vert \alpha_{n,1}\right\vert
+\left\vert \alpha_{n,2}\right\vert ), \tag{35}%
\end{equation}
then the eigenvalues \ $\lambda_{n,j}$ for $j=1,2$\ and $n\geq N$ are simple
and satisfy the asymptotic formulas (32).
\end{theorem}

\begin{proof}
$(a)$ It readily follows from the proof of Theorem 1$(a)$ that
\begin{equation}
\lambda_{n,j}-(2\pi n)^{2}=a_{j}(\left\vert q_{2n}\right\vert +\beta_{n,j}),
\tag{36}%
\end{equation}
where
\begin{equation}
-\left\vert \frac{\alpha_{n,j}}{u_{n,j}}\right\vert \leq\beta_{n,j}%
\leq\left\vert \frac{\alpha_{n,j}}{u_{n,j}}\right\vert . \tag{37}%
\end{equation}

Now, we prove that $a_{1}=-1$ and $a_{2}=1.$ To do this, we prove that the
remaining cases, namely the cases Case 1: $a_{1}=a_{2}=1,$ Case 2:
$a_{1}=a_{2}=-1$ and Case 3: $a_{1}=1,a_{2}=-1$ contradict condition (35).
Let's first consider the Case 1. Using (36) in (30) for $j=1$ and $j=2$ we
obtain
\begin{equation}
(\left\vert q_{2n}\right\vert +\beta_{n,1})u_{n,1}=q_{2n}\overline{u_{n,1}%
}+\alpha_{n,1} \tag{38}%
\end{equation}
and
\begin{equation}
(\left\vert q_{2n}\right\vert +\beta_{n,2})u_{n,2}=q_{2n}\overline{u_{n,2}%
}+\alpha_{n,2} \tag{39}%
\end{equation}
Multiplying both sides of (38) and (39) by $\overline{u_{n,2}}$ and
$-\overline{u_{n,1}}$ respectively and then summing these equalities we get
\begin{equation}
(u_{n,1}\overline{u_{n,2}}-u_{n,2}\overline{u_{n,1}})\left\vert q_{2n}%
\right\vert =\beta_{n,2}u_{n,2}\overline{u_{n,1}}-\beta_{n,1}u_{n,1}%
\overline{u_{n,2}}+\alpha_{n,1}\overline{u_{n,2}}-\alpha_{n,2}\overline
{u_{n,1}}, \tag{40}%
\end{equation}
where
\[
\left\vert \beta_{n,2}u_{n,2}\overline{u_{n,1}}-\beta_{n,1}u_{n,1}%
\overline{u_{n,2}}\right\vert \leq\left\vert \alpha_{n,2}\overline{u_{n,1}%
}\right\vert +\left\vert \alpha_{n,1}\overline{u_{n,2}}\right\vert
\]
according (37). Now, dividing both sides of (40) by $\left\vert q_{2n}%
\right\vert $ and using (35) we obtain the proof of (33) which contradicts
(29) due to (34). In the same way we get the same contradiction for the Case 2.

It remains to consider the Case 3. In this case from (36) and (37) it follows
that
\[
\lambda_{n,1}-\lambda_{n,2}=2\left\vert q_{2n}\right\vert +\beta_{n,1}%
+\beta_{n,2}\geq2\left\vert q_{2n}\right\vert -\left\vert \frac{\alpha_{n,2}%
}{u_{n,2}}\right\vert -\left\vert \frac{\alpha_{n,1}}{u_{n,1}}\right\vert .
\]
Therefore using (35) and (29), we obtain $\lambda_{n,1}-\lambda_{n,2}>0$ which
contradicts to the assumption $\lambda_{n,1}\leq\lambda_{n,2}$ (see (5)). The
theorem is proved.
\end{proof}

Now, replacing condition (35) with more verifiable and applicable conditions,
we obtain the following elegant corollary of Theorems 1 and 2.

\begin{corollary}
$(a)$ If
\begin{equation}
q_{2n}=o\left(  n^{-1}\right)  , \tag{41}%
\end{equation}
then the eigenvalues \ $\lambda_{n,j}$ for $j=1,2$\ satisfy the asymptotic
formulas
\begin{equation}
\lambda_{n,j}=(2\pi n)^{2}+o\left(  n^{-1}\right)  . \tag{42}%
\end{equation}

$(b)$ Let $\mathbb{N}_{1}(\varepsilon)$ be the subset of the set $\left\{
n\in\mathbb{N}:\text{ }n>N\right\}  $ such that
\begin{equation}
\left\vert q_{2n}\right\vert \geq\frac{\varepsilon}{n}, \tag{43}%
\end{equation}
for $n\in\mathbb{N}_{1}(\varepsilon),$ where $\varepsilon$ is a fixed positive
number and $N$ is a sufficiently large positive integer. Then the eigenvalues
\ $\lambda_{n,j}$ for $n\in\mathbb{N}_{1}(\varepsilon)$ and $j=1,2$\ are
simple and satisfy asymptotic formulas (32).
\end{corollary}

\begin{proof}
Proof $(a)$ follows directly from Theorem 1$(a)$. If (43) is satisfied, then
(35) is true, since
\[
\left\vert \alpha_{n,1}\right\vert +\left\vert \alpha_{n,2}\right\vert
=o\left(  n^{-1}\right)
\]
(see (30)). Therefore, proof $(b)$ follows from Theorem 2.
\end{proof}

Instead (6) using the equality
\[
(\mu_{n,j}-(2\pi n+\pi)^{2})(\Phi_{n,j}(x),e^{i\pi(2n+1)x})=(q(x)\Phi
_{n,j}(x),e^{i\pi(2n+1)x}),
\]
where $\Phi_{n,j}(x)$ is normalized eigenfunction of $L_{\pi}(q)$
corresponding to the eigenvalue $\mu_{n,j},$ and arguing as in the proof of
(30) we obtain
\[
(\mu_{n,j}-(2\pi n+\pi)^{2})v_{n,j}=q_{2n+1}\overline{v_{n,j}}+\widetilde
{\alpha}_{n,j}%
\]
where
\[
v_{n,j}=(\Phi_{n,j}(x),e^{i\pi(2n+1)x}),\text{ }\widetilde{\alpha}%
_{n,j}=o(n^{-1}).
\]
Using these formulas instead of (30) and reasoning in the same way as in the
proof of Theorems 1 and 2 and Corollary 1 we obtain the following theorem and corollary

\begin{theorem}
If $q\in L_{1}[0,1],$ then

$(a)$ The eigenvalues $\mu_{n,1}$ and $\mu_{n,2}$ satisfy the equalities
\[
\mu_{n,j}-(2\pi n+\pi)^{2}=a_{j}\left\vert q_{2n+1}\right\vert +o(n^{-1})
\]
for $j=1,2$ as $n\rightarrow\infty,$ where $a_{j}$ is either $1$ or $-1$.

$(b)$ If
\[
\left\vert q_{2n+1}\right\vert \geq2\left(  \left\vert \widetilde{\alpha
}_{n,1}\right\vert +\left\vert \widetilde{\alpha}_{n,2}\right\vert \right)
\]
for $n\geq N,$ then the eigenvalues \ $\mu_{n,j}$ for $j=1,2$\ and $n\geq N$
are simple and satisfy the asymptotic formulas
\begin{equation}
\mu_{n,j}=(2\pi n+\pi)^{2}+(-1)^{j}\left\vert q_{2n+1}\right\vert +o\left(
n^{-1}\right)  \tag{44}%
\end{equation}
as $n\rightarrow\infty.$
\end{theorem}

\begin{corollary}
$(a)$ If
\[
q_{2n+1}=o\left(  n^{-1}\right)  ,
\]
then the eigenvalues \ $\mu_{n,j}$ for $j=1,2$\ satisfy the asymptotic
formulas
\[
\mu_{n,j}=(2\pi n+\pi)^{2}+o\left(  n^{-1}\right)  .
\]

$(b)$ Let $\mathbb{N}_{2}(\varepsilon)$ be the subset of the set $\left\{
n\in\mathbb{N}:\text{ }n>N\right\}  $ such that
\[
\left\vert q_{2n+1}\right\vert \geq\frac{\varepsilon}{n},
\]
for $n\in\mathbb{N}_{2}(\varepsilon),$ where $\varepsilon$ is a fixed positive
number and $N$ is a sufficiently large positive integer. Then the eigenvalues
\ $\mu_{n,j}$ for $n\in\mathbb{N}_{2}(\varepsilon)$ and $j=1,2$\ are simple
and satisfy asymptotic formulas (44).
\end{corollary}

Now let us consider the eigenfunctions of the periodic and antiperiodic problems.

\begin{theorem}
If (43) holds, then the eigenfunction $\varphi_{n,j}(x)$ corresponding to
$\lambda_{n,j}$ can be chosen so that
\begin{equation}
\varphi_{n,j}(x)=e^{i2\pi nx}+\gamma_{2n,j}e^{-i2\pi nx}+O(n^{-1}) \tag{45}%
\end{equation}
for $j=1,2,$ where
\[
\gamma_{2n,j}=\frac{(-1)^{j}\left\vert q_{2n}\right\vert }{q_{2n}}+o(1).
\]

\end{theorem}

\begin{proof}
By Corollary 1, $\lambda_{n,1}$ satisfies (32) for $j=1.$ Therefore, from
equality (30) we obtain that
\[
-\left\vert q_{2n}\right\vert u_{n,1}=q_{2n}\overline{u_{n,1}}+o(n^{-1}),
\]
This equality with (43) and (29) implies that
\begin{equation}
\frac{\overline{u_{n,1}}}{u_{n,1}}=\frac{-\left\vert q_{2n}\right\vert
}{q_{2n}}+o(1) \tag{46}%
\end{equation}
Denoting
\[
\gamma_{2n,j}=\frac{\overline{u_{n,j}}}{u_{n,j}},\text{ }\varphi
_{n,j}(x)=\frac{\Psi_{n,j}(x)}{u_{n,j}},
\]
using\ (28) and (46) we see that (45) holds for $j=1$.

Taking into account that $\lambda_{n,2}$ satisfies (32) for $j=2$ in the same
way we get the proof of (45) for $j=2$.
\end{proof}

Instead Corollary 1 using Corollary 2 and arguing as in the proof of Theorem 4
\ we get

\begin{theorem}
If the condition of Corollary 2 $(b)$ is satisfied, then the eigenfunction
$\phi_{n,j}(x)$ corresponding to $\mu_{n,j}$ can be chosen so that
\begin{equation}
\phi_{n,j}(x)=e^{i\left(  2\pi n+\pi\right)  x}+\gamma_{2n+1,j}e^{-i(2\pi
n+\pi)x}+O(n^{-1}) \tag{47}%
\end{equation}
for $j=1,2,$ where
\[
\gamma_{2n+1,j}=\frac{(-1)^{j}\left\vert q_{2n+1}\right\vert }{q_{2n+1}%
}+o(1).
\]

\end{theorem}

\begin{remark}
Since $\frac{\left\vert q_{2n}\right\vert }{q_{2n}}=e^{-i\alpha_{2n}},$ where
$\alpha_{2n}=\arg q_{2n},$ formula (45) can be written in the form
\[
\varphi_{n,j}(x)=e^{i2\pi nx}+(-1)^{j}e^{-i\alpha_{2n}}e^{-i2\pi nx}+o(1).
\]
Multiplying by $e^{\frac{1}{2}i\alpha_{2n}}$ we obtain an eigenfunctions
\[
e^{\frac{1}{2}i\alpha_{2n}}\varphi_{n,j}(x)=e^{\frac{1}{2}i\alpha_{2n}%
}e^{i2\pi nx}+(-1)^{j}e^{-\frac{1}{2}i\alpha_{2n}}e^{-i2\pi nx}+o(1),
\]
where
\[
e^{\frac{1}{2}i\alpha_{2n}}e^{i2\pi nx}-e^{-\frac{1}{2}\alpha_{2n}}e^{-i2\pi
nx}=2i\sin(2\pi nx+\frac{\alpha_{2n}}{2})
\]
and
\[
e^{\frac{1}{2}i\alpha_{2n}}e^{i2\pi nx}+e^{-\frac{1}{2}\alpha_{2n}}e^{-i2\pi
nx}=2\cos(2\pi nx+\frac{\alpha_{2n}}{2})
\]
Therefore, there exist normalized eigenfunctions $\Psi_{n,1}(x)$ and
$\Psi_{n,2}(x)$ of $L_{0}(q)$ corresponding respectively to the eigenvalues
$\lambda_{n,1}$ and $\lambda_{n,2}$ which satisfy the formulas
\[
\Psi_{n,1}(x)=\sqrt{2}\sin(2\pi nx+\frac{\alpha_{2n}}{2})+o(1)
\]
and
\[
\Psi_{n,2}(x)=\sqrt{2}\cos(2\pi nx+\frac{\alpha_{2n}}{2})+o(1).
\]

In the same way, from (47) we conclude that there exist normalized
eigenfunctions $\Phi_{n,1}(x)$ and $\Phi_{n,2}(x)$ of $L_{\pi}(q)$
corresponding respectively to the eigenvalues $\mu_{n,1}$ and $\mu_{n,2}$
which satisfy the formulas
\[
\Phi_{n,1}(x)=\sqrt{2}\sin((2\pi n+\pi)x+\frac{\alpha_{2n+1}}{2})+o(1)
\]
and
\[
\Phi_{n,2}(x)=\sqrt{2}\cos((2\pi n+\pi)x+\frac{\alpha_{2n+1}}{2})+o(1),
\]
where $\alpha_{2n+1}=\arg q_{2n+1}.$
\end{remark}

Thus using (7) for $m=1$ we obtained asymptotic formulas of order $o\left(
n^{-1}\right)  .$ Now, using (7) for $m=2$ we obtain asymptotic formulas of
order $o\left(  n^{-2}\right)  $ that also has an elegant form. From (10),
Proposition 1, (16) and (17) for $s=0,$ it follows that (7) for $m=2$ can be
written in the form
\begin{equation}
(\lambda_{n,j}-(2\pi n)^{2}-A_{2}((2\pi n)^{2})u_{n,j}=(q_{2n}-S_{2n}%
+2Q_{0}Q_{2n})\overline{u_{n,j}}+o(n^{-2}) \tag{48}%
\end{equation}
where
\[
A_{2}((2\pi n)^{2}=a_{1}((2\pi n)^{2})+a_{2}((2\pi n)^{2},
\]%
\[
a_{1}((2\pi n)^{2})=\sum\limits_{k\neq0,2n}\frac{\left\vert q_{k}\right\vert
^{2}}{(2\pi n)^{2}-(2\pi(n-k))^{2}}%
\]
and
\[
a_{2}((2\pi n)^{2})=\sum\limits_{k,k+l\neq0,2n}\frac{q_{k}q_{l}q_{-k-l}%
}{[(2\pi n)^{2}-(2\pi(n-k))^{2}][(2\pi n)^{2}-(2\pi(n-k-l))^{2}]}.
\]

Instead of (30) using (48) and arguing as in the proof of Theorems 1 and 2 and
Corollary 1 we obtain the following.

\begin{theorem}
Let $\mathbb{N}_{3}(\varepsilon)$ be the subset of the set $\left\{
n\in\mathbb{N}:\text{ }n>N\right\}  $ such that
\begin{equation}
\left\vert q_{2n}-S_{2n}+2Q_{0}Q_{2n}\right\vert \geq\frac{\varepsilon}{n^{2}%
}, \tag{49}%
\end{equation}
for $n\in\mathbb{N}_{3},$ where $\varepsilon$ is a fixed positive number and
$N$ is a sufficiently large positive integer. Then the eigenvalues
\ $\lambda_{n,j}$ for $n\in\mathbb{N}_{3}$ and $j=1,2$\ are simple and satisfy
the asymptotic formulas
\begin{equation}
\lambda_{n,j}=(2\pi n)^{2}+A_{2}((2\pi n)^{2})+(-1)^{j}\left\vert
q_{2n}-S_{2n}+2Q_{0}Q_{2n}\right\vert +o(n^{-2}) \tag{50}%
\end{equation}
as $n\rightarrow\infty.$
\end{theorem}

\begin{proof}
Dividing both sides of (48) by $u_{n,j}$ and using (29) we obtain
\[
\left\vert (\lambda_{n,j}-(2\pi n)^{2}-A_{2}((2\pi n)^{2})-(q_{2n}%
-S_{2n}+2Q_{0}Q_{2n})\overline{u_{n,j}}\frac{1}{u_{n,j}}\right\vert
=o(n^{-2})
\]
and%
\[
\left\vert (\lambda_{n,j}-(2\pi n)^{2}-A_{2}((2\pi n)^{2})\right\vert
-\left\vert (q_{2n}-S_{2n}+2Q_{0}Q_{2n})\right\vert =o(n^{-2})
\]
Thus we have
\[
\lambda_{n,j}=(2\pi n)^{2}+A_{2}((2\pi n)+a_{j}\left\vert q_{2n}-S_{2n}%
+2Q_{0}Q_{2n}\right\vert +o(n^{-2}),
\]
where $a_{j}=\pm1,$ since$\lambda_{n,j}-(2\pi n)^{2}-A_{2}((2\pi n)^{2}$ is a
real number according to Lemma 1 $(b).$ Now, repeating the proof of equalities
$a_{1}=-1$ and $a_{2}=1$ which was done in the proof of \ Theorem 2 we
complete the proof of the theorem.
\end{proof}

The proof of corresponding results for antiperiodic problem can be carried out
in a similar way.

\begin{theorem}
Let $\mathbb{N}_{4}(\varepsilon)$ be the subset of the set $\left\{
n\in\mathbb{N}:\text{ }n>N\right\}  $ such that
\[
|q_{2n+1}-S_{2n+1}+Q_{0}Q_{2n+1}|\geq\varepsilon n^{-2},
\]
for $n\in\mathbb{N}_{4},$ where $\varepsilon$ is a fixed positive number and
$N$ is a sufficiently large positive integer. Then the eigenvalues
\ $\mu_{n,j}$ of the operator $L_{\pi}(q)$ for $n\in\mathbb{N}_{4}$ are simple
and satisfy the asymptotic formulas
\[
\mu_{n,j}=(2\pi n+\pi)^{2}+A_{2}((2\pi n+\pi)^{2})+
\]%
\[
(-1)^{j}\left\vert q_{2n+1}-S_{2n+1}+2Q_{0}Q_{2n+1}\right\vert +o(n^{-2})
\]
for $j=1,2.$
\end{theorem}

Now, consider the gaps in the high energy region.

\begin{theorem}
For the length of the gaps $\Delta_{n}=(\lambda_{n,1},\lambda_{n,2})$ and
$\Omega_{n}=\left(  \mu_{n,1},\mu_{n,2}\right)  $ \ the following formulas
hold
\[
\left\vert \Delta_{n}\right\vert =2\left\vert q_{2n}\right\vert +o\left(
n^{-1}\right)  ,\text{ }\left\vert \Delta_{n}\right\vert =2\left\vert
q_{2n}-S_{2n}+2Q_{0}Q_{2n}\right\vert +o(n^{-2})
\]
and
\[
\left\vert \Omega_{n}\right\vert =2\left\vert q_{2n+1}\right\vert +o\left(
n^{-1}\right)  ,\text{ }\left\vert \Omega_{n}\right\vert =2\left\vert
q_{2n+1}-S_{2n+1}+2Q_{0}Q_{2n+1}\right\vert +o(n^{-2}).
\]
where $\left\vert \Delta\right\vert $ denotes the length of the interval
$\Delta.$
\end{theorem}

\begin{proof}
Let's prove the formulas for $\left\vert \Delta_{n}\right\vert .$ Proof of
formulas for $\left\vert \Omega_{n}\right\vert $ similar. If (43) and (49) are
satisfied, then the proof follows from Corollary 1 $(b)$ and (50),
respectively. We will now present a proof without assuming that (43) and (49)
hold. It follows from (34) that
\begin{equation}
\frac{\overline{u_{n,2}}}{u_{n,2}}=-\frac{\overline{u_{n,1}}}{u_{n,1}%
}+O(n^{-2}) \tag{51}%
\end{equation}
Therefore, using (30) for $j=1,2$ and (5) we obtain
\[
\lambda_{n,2}-\lambda_{n,1}=\left\vert 2q_{2n}(\frac{\overline{u_{n,1}}%
}{u_{n,1}}+O(n^{-2}))+o(n^{-1})\right\vert =2\left\vert q_{2n}\right\vert
+o(n^{-1}).
\]
The first equality for $\left\vert \Delta_{n}\right\vert $ is proved. Instead
of (30) using (48) and repeating the above proof we get the proof of the
second equality for $\left\vert \Delta_{n}\right\vert .$
\end{proof}

Now, we derive asymptotic formulas of arbitrary order for eigenvalues of the
operators $L_{0}(q)$ and $L_{\pi}(q)$ \ for $q\in L_{1}(0,1).$ Dividing both
sides of (7) by $u_{n,j}$ and repeating the proof of (32) and using (10) we
obtain
\[
\lambda_{n,j}=(2\pi n)^{2}+A_{m}(\lambda_{n,j})+a_{j}\left\vert q_{2n}%
+B_{m}(\lambda_{n,j})\right\vert +O((\frac{\ln n}{n})^{m+1}),
\]
where $a_{j}$ is either $1$ or $-1.$ If there exists $\varepsilon>0$ such
that
\begin{equation}
\left\vert q_{2n}+B_{m}(\lambda_{n,j})\right\vert \geq\frac{\varepsilon}%
{n^{m}} \tag{52}%
\end{equation}
for $n\geq N,$ then, repeating the proof of the relation $a_{1}=-1$ and
$a_{2}=1$ which was done in the proof of Theorem 2, we obtain that
\begin{equation}
\lambda_{n,j}=(2\pi n)^{2}+A_{m}(\lambda_{n,j})+(-1)^{j}\left\vert
q_{2n}+B_{m}(\lambda_{n,j})\right\vert +O((\frac{\ln n}{n})^{m+1}) \tag{53}%
\end{equation}
for $m=0,1,...$, where $A_{0}(\lambda_{n,j})=B_{0}(\lambda_{n,j})=0,$
$A_{m}(\lambda_{n,j})$ and $B_{m}(\lambda_{n,j})$ for $m=1,2,...$ are defined
in (7). Now, we remove $\lambda_{n,j}$ from the right side of (53) and obtain
the asymptotic formulas
\begin{equation}
\lambda_{n,j}=E_{n,j,m}+O((\frac{\ln n}{n})^{m+1}) \tag{54}%
\end{equation}
for $m=0,1,...$, where the expression $E_{n,m,j}$ contains only Fourier
coefficients and are defined successively as follows
\[
E_{n,j,0}=(2\pi n)^{2},\text{ }E_{n,j,1}=(2\pi n)^{2}+A_{1}(E_{n,j,0}%
)^{2}+(-1)^{j}\left\vert q_{2n}+B_{1}(E_{n,j,0})\right\vert ,
\]
and%
\[
E_{n,j,k}=(2\pi n)^{2}+A_{k}(E_{n,j,k-1})^{2}+(-1)^{j}\left\vert q_{2n}%
+B_{k}(E_{n,k-1,j})\right\vert .
\]

\begin{remark}
Note that formulas (32) and (50) are more accurate formulas than (54) for
$m=0,1.$ Therefore in the next theorem we obtain asymptotic formulas for the
cases $m\geq2.$ Using Lemma 2, it is easy to check that (43)$\Longrightarrow
$(49)$\Longrightarrow$(52) for $m\geq2$.
\end{remark}

\begin{theorem}
$(a)$ If (52) is satisfied, then the eigenvalue $\lambda_{n,j}$ satisfies
formula (54) for $m=2,3,...$and the following estimate is valid%
\[
\left\vert \Delta_{n}\right\vert =E_{n,2,m}-E_{n,1,m}+O((\frac{\ln n}%
{n})^{m+1}).
\]
Moreover, if at least one of the inequalities (43) and (49) is satisfied, then
this theorem continues to be valid.

$(b)$ If $q_{2n}+B_{m}(\lambda_{n,j})=O(n^{-m})$ and $q_{2n}+B_{m}%
(\lambda_{n,j})=o(n^{-m})$ respectively, then $\left\vert \Delta
_{n}\right\vert =O(n^{-m})$ and $\left\vert \Delta_{n}\right\vert =o(n^{-m})$
\end{theorem}

\begin{proof}
$(a)$ We prove (54) by induction. It is proved for $m=1$ (see Corollary 1$(b)$
and Theorem 6). Assume that (54) is true for $m=k-1$. Substituting the value
\[
E_{n,j,k-1}+O((\frac{\ln n}{n})^{k})
\]
of $\lambda_{n,j},$ given by (54) for $m=k-1$ in (53) for $\ m=k$ we obtain
\[
\lambda_{n,j}=(2\pi n)^{2}+A_{k}(E_{n,j,k-1}+O((\frac{\ln n}{n})^{k}))+
\]%
\[
(-1)^{j}\left\vert q_{2n}+B_{m}(E_{n,j,k-1}+O((\frac{\ln n}{n})^{k}%
))\right\vert +O((\frac{\ln n}{n})^{k+1}).
\]
Now, to prove (54) for $m=k,$ it remains to show that
\[
A_{k}(E_{n,j,k-1}+O((\frac{\ln\left\vert n\right\vert }{n})^{k}))=A_{k}%
(E_{n,j,k-1})+O((\frac{\ln\left\vert n\right\vert }{n})^{k+1})
\]
and%
\[
B_{k}(E_{n,j,k-1}+O((\frac{\ln\left\vert n\right\vert }{n})^{k}))=B_{k}%
(E_{n,j,k-1})+O((\frac{\ln\left\vert n\right\vert }{n})^{k+1}).
\]
These equalities can be proven by reasoning in the same way as in the proof of
Proposition 1. The estimate for $\left\vert \Delta_{n}\right\vert $ follows
from (54). The second statement of the theorem follows from Remark 2.

$(b)$ Arguing as in the proof of Proposition 1 one can easily verify that the
derivative of $A_{m}$ at point $\lambda\in(\lambda_{n,1},\lambda_{n,2})$ is
$O(n^{-2}).$ Therefore, by mean-value theorem we have
\[
A_{m}(\lambda_{n,2})-A_{m}(\lambda_{n,1})=(\lambda_{n,2}-\lambda
_{n,1})O(n^{-2}).
\]
Thus, from (7) it follows that
\begin{equation}
(\lambda_{n,2}-\lambda_{n,1})(1+O(n^{-2}))=(q_{2n}+B_{m}(\lambda_{n,2}%
))\frac{\overline{u_{n,2}}}{u_{n,2}}- \tag{55}%
\end{equation}%
\[
(q_{2n}+B_{m}(\lambda_{n,1}))\frac{\overline{u_{n,1}}}{u_{n,1}}+O((\frac
{\ln\left\vert n\right\vert }{n})^{m+1})
\]
that gives the proof of $(b).$
\end{proof}

In a similar way we obtain similar asymptotic formulas for the eigenvalues
$\mu_{n,1}$ and $\mu_{n,2}$\ of $L_{\pi}(q)$ and for the length $\left\vert
\Omega_{n}\right\vert =\mu_{n,2}-\mu_{n,1}$ of the gap $(\mu_{n,1},\mu
_{n,2}).$

It follows from (55) and (51) that
\[
\lambda_{n,2}-\lambda_{n,1}=(2q_{2n}+B_{m}(\lambda_{n,2})+B_{m}(\lambda
_{n,1}))\frac{\overline{u_{n,2}}}{u_{n,2}}(1+O(n^{-2}))+
\]%
\[
(q_{2n}+B_{m}(\lambda_{n,1}))O(n^{-2})+O((\frac{\ln\left\vert n\right\vert
}{n})^{m+1}).
\]
using this equality, (16) and (17) and repeating the proof of Theorem 8 we obtain

\begin{theorem}
If (15) is satisfied then
\[
\left\vert \Delta_{n}\right\vert =2\left\vert q_{2n}-S_{2n}+2Q_{0}%
Q_{2n}\right\vert +o(n^{-s-2})
\]
and
\[
\left\vert \Omega_{n}\right\vert =2\left\vert q_{2n+1}-S_{2n+1}+2Q_{0}%
Q_{2n+1}\right\vert +o(n^{-s-2}).
\]

\end{theorem}

Now we consider the eigenfunctions $\varphi_{n,j}(x)=\frac{\Psi_{n,j}%
(x)}{u_{n,j}}.$ Assume that (43) holds. Writing the decomposition of
$\Psi_{n,j}(x)$ by the orthonormal basis $\{e^{i2\pi(n-n_{1})x}:n_{1}%
\in\mathbb{Z}\}$ we obtain
\[
\Psi_{n,j}(x)-u_{n,j}e^{i2\pi nx}=\sum_{\substack{n_{1}=-\infty,\\n_{1}\neq
0}}^{\infty}(\Psi_{n,j}(x),e^{i2\pi(n-n_{1})x})e^{i2\pi(n-n_{1})x}.
\]

The right-hand side of this formula can be obtained from the right-hand side
of (6) by replacing $q_{n_{1}}$ with $e^{i2\pi(n-n_{1})x}.$ Since we obtained
(7) from (6) by iteration, doing the same iterations and the same estimations
we obtain
\[
\Psi_{n,j}(x)=u_{n,j}e^{i2\pi nx}+\overline{u_{n,j}}e^{-i2\pi nx}+
\]%
\[
u_{n,j}A_{m}^{\ast}(\lambda_{n,j})+\overline{u_{n,j}}B_{m}^{\ast}%
(\lambda_{n,j})+O((\frac{\ln\left\vert n\right\vert }{n})^{m+1}),
\]
where $A_{m}^{\ast}$ and $B_{m}^{\ast}$ are obtained from $A_{m}$ and $B_{m}$
respectively by replacing $q_{n_{1}}$ with $e^{i2\pi(n-n_{1})x}.$ Dividing
both sides of this equality by $u_{n,j}$ we obtain
\[
\varphi_{n,j}(x)=e^{i2\pi nx}+A_{m}^{\ast}(\lambda_{n,j})+
\]%
\[
+\frac{\overline{u_{n,j}}}{u_{n,j}}(e^{-i2\pi nx}+B_{m}^{\ast}(\lambda
_{n,j}))+O((\frac{\ln n}{n})^{m+1}),
\]
where $\frac{\overline{u_{n,j}}}{u_{n,j}}$ can be estimated as follows. It
follows from (7) that
\[
\frac{\overline{u_{n,j}}}{u_{n,j}}=\frac{(\lambda_{n,j}-(2\pi n)^{2}%
-A_{m}(\lambda_{n,j}))}{q_{2n}+B_{m}(\lambda_{n,j})}+O\left(  \frac{1}%
{q_{2n}+B_{m}(\lambda_{n,j})}(\frac{\ln\left\vert n\right\vert }{n}%
)^{m+1}\right)  .
\]
On the other hand, by (16)-(17) and (43) there exist a positive constant
$c_{3}$ such that
\[
\frac{1}{\left\vert q_{2n}+B_{m}(\lambda_{n,j})\right\vert }\leq c_{3}n.
\]
Therefore the last equality can by written in the form
\[
\frac{\overline{u_{n,j}}}{u_{n,j}}=\frac{(\lambda_{n,j}-(2\pi n)^{2}%
-A_{m}(\lambda_{n,j}))}{q_{2n}+B_{m}(\lambda_{n,j})}+O\left(  n(\frac
{\ln\left\vert n\right\vert }{n})^{m+1}\right)  .
\]
Now, using (53) and the last two equalities of the proof of Theorem 9$(a)$ we
get
\[
\frac{\overline{u_{n,j}}}{u_{n,j}}=\frac{(E_{n,j,m}-(2\pi n)^{2}%
-A_{m}(E_{n,j,m}))}{(q_{2n}+B_{m}(E_{n,j,m}))}+O\left(  n(\frac{\ln\left\vert
n\right\vert }{n})^{m+1}\right)  .
\]
Moreover, reasoning in the same way as when proving the last two equalities in
the proof of Theorem 9$(a)$, it is easy to verify that%
\[
A_{m}^{\ast}(\lambda_{n,j})=A_{m}^{\ast}(E_{n,j,m})+O((\frac{\ln\left\vert
n\right\vert }{n})^{m+1})
\]
and%
\[
B_{m}^{\ast}(\lambda_{n,j})=B_{m}^{\ast}(E_{n,j,m})+O((\frac{\ln\left\vert
n\right\vert }{n})^{m+1}).
\]
Thus we have the following formula for the eigenfunction $\varphi_{n,j}(x)$ of
$L_{0}(q)$%
\[
\varphi_{n,j}(x)=e^{i2\pi nx}+A_{m}^{\ast}(E_{n,j,m})+
\]%
\[
\frac{(E_{n,j,m}-(2\pi n)^{2}-A_{m}(E_{n,j,m}))}{(q_{2n}+B_{m}(E_{n,j,m}%
))}(e^{-i2\pi nx}+B_{m}^{\ast}(E_{n,j,m}))+h_{m}(x),
\]
for $m=2,3,...$, where $\left\Vert h_{m}\right\Vert =O(n(\frac{\ln\left\vert
n\right\vert }{n})^{m+1}).$ In the similar way we obtain the similar
asymptotic formulas for the eigenfunctions of $L_{\pi}(q)$.

\section{On the Kronig-Penney model}

Now, we take the Kronig-Penney model as an example and use this example to
illustrate the obtained results. The Kronig-Penney model is a simplified model
of the electron in a one-dimensional periodic potential and has been studied
in many works (see, for example, [5], [7, Chap.3] and [13, Chap.21]). In the
case of the Kronig-Penney model, the potential $q(x)$ has the form
\[
q(x)=\left\{
\begin{tabular}
[c]{l}%
$a\text{ if }x\in\lbrack0,c]$\\
$b\text{ if }x\in(c,d]$%
\end{tabular}
\ \ \ \ \ \ \ \ \ \ \ \ \ \ \ \ \ \ \ \ \right.
\]
and $q(x+d)=q(x)$ where $c\in(0,d).$ For simplicity of notation and without
loss of generality, we assume that $d=1,$ $a<b$ and (2) is satisfied. Then we
have
\begin{align}
q(x)  &  =\left\{
\begin{tabular}
[c]{l}%
$a\text{ if }x\in\lbrack0,c]$\\
$b\text{ if }x\in(c,1]$%
\end{tabular}
\ \ \ \ \ \ \ \ \ \ \ \ \ \ \ \ \ \ ,\ \ \right. \tag{56}\\
a  &  <0<b,\text{ }ac+(1-c)b=0.\nonumber
\end{align}
Let us calculate the first and second terms of the asymptotic formulas for the
eigenvalues and gaps in the spectrum of $L(q)$ when $q$ is defined by formula
(56). To write the gap length in a compact form and for the brevity of some
calculation, it is convenient to introduce the following notations%
\[
\delta_{k}=\left\{
\begin{tabular}
[c]{l}%
$\left\vert \Delta_{n}\right\vert \text{ if }k=2n$\\
$\left\vert \Omega_{n}\right\vert \text{ if }k=2n+1$%
\end{tabular}
\ \ \ \ \ \ \ \ \ \ \ \ \ \ \ \ \ \ ,\ \ \right.
\]
where $n\geq N$ and $\ N$ is sufficiently large positive number. By Theorem 8
we have
\begin{equation}
\delta_{k}=2\left\vert q_{k}-S_{k}+2Q_{0}Q_{k}\right\vert +o(k^{-2}). \tag{57}%
\end{equation}
Thus $\delta_{k}$ for even $\ k=2n$ and odd $k=2n+1,$ respectively, is the
length of the gap lying in the vicinity of the periodic and antiperiodic
eigenvalues $(2\pi n)^{2}$ and $((2n+1)\pi)^{2}$ of the unperturbed operator
$L(0).$ According to (57), to estimate the length of the gaps we need to
consider the Fourier coefficients of the functions
\[
q(x),\text{ }Q(x)=\int_{0}^{x}q(t)\,dt,\quad S(x)=Q^{2}(x).
\]
From (56) it follows that
\begin{equation}
q_{k}=\int\limits_{0}^{c}ae^{-2\pi kix}dx+\int\limits_{c}^{1}be^{-2\pi
kix}dx=\frac{a-b}{2\pi ki}(1-e^{-2\pi kic}), \tag{58}%
\end{equation}%
\begin{equation}
Q(x)=\left\{
\begin{tabular}
[c]{l}%
$ax\text{ if }x\in\lbrack0,c],$\\
$bx-b\text{ if }x\in(c,1],$%
\end{tabular}
\ \ \ \ \ \ \ \ \ \ \ \ \ \ \ \ \ \ \ \ \ \right.  , \tag{59}%
\end{equation}
and%
\begin{equation}
S(x)=\left\{
\begin{tabular}
[c]{l}%
$a^{2}x^{2}\text{ if }x\in\lbrack0,c],$\\
$(bx-b)^{2}\text{ if }x\in(c,1],$%
\end{tabular}
\ \ \ \ \ \ \ \ \ \ \ \ \ \ \ \ \ \ \ \ \ \right.  . \tag{60}%
\end{equation}
Now, let us calculate the Fourier coefficients
\[
Q_{k}=\int\limits_{0}^{1}Q(x)e^{-2\pi ikx}dx\
\]
\ and
\[
S_{k}=\int\limits_{0}^{1}S(x)e^{-2\pi ikx}dx
\]
of $Q$ and $S$. From (59) it follows that
\[
Q_{0}=\int\limits_{0}^{1}Q(x)\,dx=\int\limits_{0}^{c}ax\,dx+\int
\limits_{c}^{1}bx-b\,dx=\frac{1}{2}(a-b)c^{2}+\frac{b}{2}-b(1-c).
\]
Now, using the equality $(a-b)c=-b$ (see (56)) we get
\[
Q_{0}=\frac{1}{2}b(c-1).
\]
From (59) and (60) by direct calculation one can find $Q_{k}$ for $k\neq
0$\ and $S_{k}$. We calculate these Fourier coefficients using (58) and the
following relations%
\[
Q(1)=\int_{0}^{1}q(x)\,dx=0=Q(0),\text{ }Q^{^{\prime}}(x)=q(x)
\]
and $S(1)=S(0)=0,$ $S^{^{\prime}}(x)=2Q(x)q(x).$ Thus, we have%
\begin{equation}
Q_{k}=\frac{q_{k}}{2\pi ki}=\frac{a-b}{(2\pi k)^{2}}(e^{-2\pi kic}-1) \tag{61}%
\end{equation}
and%

\begin{equation}
S_{k}=\frac{1}{2\pi ki}\int\limits_{0}^{1}2Q(x)q(x)e^{-2\pi ikx}\,dx, \tag{62}%
\end{equation}
where by (56) and (59)
\[
\int\limits_{0}^{1}Q(x)q(x)e^{-2\pi ikx}\,dx=(\int\limits_{0}^{c}%
a^{2}xe^{-2\pi ikx}\,dx+\int\limits_{c}^{1}b^{2}xe^{-2\pi ikx}\,dx-\int
\limits_{c}^{1}b^{2}e^{-2\pi ikx}\,dx),
\]%
\[
\int\limits_{0}^{c}xe^{-2\pi ikx}\,dx=\frac{(e^{-2\pi ikc}-1)}{(2\pi k)^{2}%
}\,-\frac{ce^{-2\pi ikc}}{2\pi ik},
\]%
\[
\int\limits_{c}^{1}xe^{-2\pi ikx}\,dx=\frac{(1-e^{-2\pi ikc})}{(2\pi k)^{2}%
}\,+\frac{ce^{-2\pi ikc}-1}{2\pi ik}%
\]
and%
\[
\int\limits_{c}^{1}e^{-2\pi ikx}\,dx=\frac{e^{-2\pi ikc}-1}{2\pi ik}.
\]
Using the last four equalities in (62) and then using the $O(k^{-3})$
notation, we obtain
\[
S_{k}=\frac{a^{2}}{\pi ki}(\frac{e^{-2\pi ikc}-1}{(2\pi k)^{2}}\,-\frac
{ce^{-2\pi ikc}}{2\pi ik})+
\]%
\[
\frac{b^{2}}{\pi ki}(\frac{(1-e^{-2\pi ikc})}{(2\pi k)^{2}}\,+\frac{ce^{-2\pi
ikc}-1}{2\pi ik})-\frac{b^{2}}{\pi ki}(\frac{e^{-2\pi ikc}-1}{2\pi ik})=
\]%
\[
\frac{e^{-2\pi ikc}}{(2\pi k)^{2}}(2a^{2}c-2b^{2}c+2b^{2})+O(k^{-3}).
\]
Moreover, using the equality $ac-bc=-b$ we get%
\[
2a^{2}c-2b^{2}c+2b^{2}=-2ab,\text{ }S_{k}=\frac{-2ba}{(2\pi k)^{2}}e^{-2\pi
ikc}+O(k^{-3}).
\]
Now using this (58) and (61), we obtain%

\begin{equation}
q_{k}-S_{k}+2Q_{0}Q_{k}=\frac{a-b}{2\pi ki}(1-e^{-2\pi kic})+\frac{2ba}{(2\pi
k)^{2}}e^{-2\pi ikc}+ \tag{63}%
\end{equation}%
\[
\frac{e^{-2\pi ikc}-1}{(2\pi k)^{2}}b(a-b)(c-1)+O(k^{-3}).
\]

It is easy to see that there are two different cases: $e^{-2\pi kic}\neq1$ and
$e^{-2\pi kic}=1,$ which depend on choice $c.$ This is easy to take into
account for rational $c$.

\begin{theorem}
Let $c=\frac{p}{m}\in(0,1),$ where $p$ and $m$ are irreducible positive integers.

$(a)$ \ If $k=ms$ for some $s\in\mathbb{N},$ then
\[
\delta_{k}=\frac{-2ba}{(2\pi k)^{2}}+o(\frac{1}{k^{2}})\sim\frac{1}{k^{2}}.
\]

$(b)$ If $k=ms+l,$ where $l=1,2,...,m-1,$ then
\[
\delta_{k}=\left\vert \frac{a-b}{2\pi k}(1-e^{-2\pi kic})\right\vert
+O(\frac{1}{k^{2}})\sim\frac{1}{k}%
\]

\end{theorem}

\begin{proof}
$(a)$ In this case
\[
e^{-2\pi kic}=e^{-2\pi psi}=1.
\]
Therefore, the proof follows from (57) and (63).

$(b)$ In this case
\[
e^{-2\pi kic}=e^{-2\pi i\frac{p}{m}(ms+l)}=e^{-2\pi i\frac{pl}{m}}.
\]
Consider $\frac{pl}{m}.$ Since $p$ and $m$ are irreducible positive integers,
none of the factors $m_{1},m_{2},...,m_{j}$ of $m$ is a factors of $p.$
Moreover, all of these factors cannot be factors of $l$ at the same time,
since $l<m.$ Therefore $\frac{pl}{m}=s+\frac{v}{m},$ where $v\in\left\{
1,2,...,m-1\right\}  .$ This implies that%
\[
e^{-2\pi i\frac{pl}{m}}-1\sim1
\]
and
\[
\frac{a-b}{2\pi ki}(1-e^{-2\pi kic})\sim\frac{1}{k}.
\]
On the other hand, remaining terms in the right side of (63) are equal to
$O(\frac{1}{k^{2}}).$Thus, the proof follows from (57) and (63).
\end{proof}

Using this theorem, consider the following interesting example.

\begin{example}
Let $c=1/2.$ Then $\left\vert \Delta_{k}\right\vert \sim\frac{1}{k^{2}}$ and
$\left\vert \Omega_{k}\right\vert \sim\frac{1}{k}.$
\end{example}

Thus, in this case, the lengths of the gaps lying in the vicinity of the
periodic and antiperiodic eigenvalues $(2\pi k)^{2}$ and $(2\pi k+\pi)^{2},$
respectively, are of the order of $\frac{1}{k^{2}}$ and $\frac{1}{k}.$

\end{document}